\documentclass{amsart}
\usepackage{amsmath}
\usepackage{amsfonts}
\usepackage{amssymb}
\usepackage{graphicx}
\usepackage{color}
\usepackage{amsmath}
\usepackage{amsfonts}
\usepackage{amssymb}
\usepackage{graphicx}
\usepackage{tikz-cd}
\usepackage{color}
\usepackage{mathbbol}
\usepackage{bbold}

\usepackage{tikz,tikz-cd}
\usetikzlibrary{matrix, positioning, decorations.pathreplacing, shapes,automata,arrows,calc}
\usepackage{caption}
\usepackage{subcaption}

\usepackage{amsthm}

\providecommand{\U}[1]{\protect\rule{.1in}{.1in}}
\newtheorem{theorem}{Theorem}[section]

\newtheorem{definition}[theorem]{Definition}

\newtheorem{lemma}[theorem]{Lemma}

\newtheorem{proposition}[theorem]{Proposition}
\newtheorem{remark}[theorem]{Remark}

\newcommand{\Z}{\mathbb{Z}}

\newcommand{\N}{\mathbb{N}}

\newcommand{\G}{\mathcal{G}}

\numberwithin{equation}{section}

\title{The RFD property for graph C*-algebras}
\author{Guillaume Bellier}

\begin{document}

\maketitle

\begin{abstract} It is proved that the C*-algebra of a graph is residually finite-dimensional (RFD) if and only if the graph has no infinite receiver, no cycle with an exit, no infinite backward chain and from each vertex, there is a finite path to a sink or a cycle or an infinite emitter. 
\end{abstract}

\section{Introduction}

A C*-algebra is residually finite-dimensional (RFD) if it has a separating family of finite-dimensional representations. 
RFD C*-algebras play central role in C*-theory and  appear at the heart of some long-standing problems. E.g.  Connes Embeddig Problem  
can be reformulated as a question of whether $C*(F_2 \times F_2)$ is RFD \cite{kirchberg1993non},  the UCT problem can be reduced to the case

This paper studies the RFD property for graph C*-algebras.  Graph C*-algebras appeared in 90-s (\cite{kumjian1997graphs}, \cite{KPR98}) as a generalization of Cuntz algebras and Cuntz-Krieger algebras and turn out to be a class of C*-algebras that is both large and tractable. It is often possible to characterize
C*-algebraic properties of $C^*(G)$ in terms of corresponding graph properties for $G$.

In \cite{bellier2025rfd} the author and T. Shulman characterized unital graph C*-algebras with the RFD-property:

a unital graph C*-algebra is RFD if and only if no cycle has an exit. 

We note that the technique used in \cite{bellier2025rfd} does not work  for non-unital graph C*-algebras. Thus to solve the problem for general graph C*-algebras one needs a completely different approach.  In this paper we use a groupoid model for graph C*-algebras (\cite{brownlowe2017graph}, \cite{webster2014path}) to give a complete characterization of when a graph C*-algebra is RFD.

\bigskip
 
 {\bf Theorem \;} {\it   Let $G$ be a graph. Then

    $C^*(G)$ is RFD if and only if $G$ has all of the following properties
    
    \begin{itemize}
        \item[a)] $G$ has no infinite receiver
        \item[b)] $G$ has no cycle with an exit
        \item[c)] $G$ has no infinite backward chain
        \item[d)] for each $v\in G^0$, there is a finite path from $v$ to a sink or a cycle or an infinite emitter.
    \end{itemize}}
 
 \bigskip

\medskip 

\noindent {\bf Acknowledgments. }
The author is grateful to Soren Eilers for useful information on groupoid model for graph $C^*$-algebras.
  The results of this article are part of the PhD project of the author.

\section{Preliminaries}

\subsection{Graph C*-algebras}

A directed graph $G = (G^0, G^1, s, r)$  consists of two countable sets $G^0$ and $G^1$ and functions $r, s: G^1\to G^0$. The elements
of $G^0$ are called {\it vertices} and the elements of $G^1$ are called {\it edges}. For each edge $e$, $s(e)$ is the {\it source} of $e$ and $r(e)$ is the {\it range } of $e$.

\medskip

The {\it graph C*-algebra} $C^*(G)$ associated with a graph $G$  is the universal $C^*$-algebra with generators $\{p_v, s_e\;|\; v\in G^0, e\in G^1\}$, where $\{p_v, v\in G^0\}$ is a collection of mutually orthogonal projections, $\{s_e, e\in G^1\}$ is
 a collection of partial isometries with mutually orthogonal ranges, and the following three
 relations are satisfied:

\medskip

\begin{enumerate}





\item $s_e^*s_e = p_{r(e)}$, for each  $e\in G^1$,

\medskip

\item $p_v = \sum_{e\in s^{-1}(v)} s_es_e^*$, for each $v\in G^0$ such that $0< \sharp(s^{-1}(v)) < \infty$,

\medskip

\item $s_e^*s_e \le p_{s(e)}$, for each  $e\in G^1$.

\end{enumerate}

\medskip

 The relations above are {\it Cuntz-Krieger relations} (CK-relations, for short). In the litterature, there are two notations for the graph $C^*$-algebras depending on the direction of the partial isometries associated with the edges. Tomforde calls classical convention when partial isometries and edges go on opposite directions and Raeburn convention when they go on the same direction. We follow the classical convention.

$C^*(G)$ is unital precisely when $G^0$ is finite. In this case $\sum_{v\in G^0}p_v$ is the unit of $C^*(G)$.

Throughout this paper we will use same notation  (usually, p) for a vertex and the projection associated with it, and we will use  same notation (usually, e) for an edge and the partial isometry associated with it.

\medskip

A {\it path} in $G$ is a sequence of edges $\nu = \nu_1 \ldots \nu_n$ such that $ r(\nu_i)= s(\nu_{i+1}) $ for each $1\le i\le n$. We write $|\nu| = n$ for the length of $\nu$. We regard vertices as paths of length $0$. We denote by $G^n$ the set of paths of length $n$ and write $G^* = \cup_{n\geq 0} G^n$ the set of all finite paths on $G$.

 We extend the range and source maps to $G^*$ by setting $s(\nu) = s(\nu_1)$ and $r(\nu) = r(\nu_{|\nu|})$ for $|\nu|>1$ and $r(v) = v= s(v)$ for $v\in G^0$. 
 


We say that a vertex $t$ is a {\it source} if it does not receive any edges, that is  $r^{-1}(t) = \emptyset$ and a {\it sink} if it does not emit any edges, that is  $s^{-1}(t) = \emptyset$.

A {\it cycle} $\mu=(\mu_n,\ldots,\mu_1)$ in $G$, is a path such that $n\geq 1$, $s(\mu) = r(\mu)$ and $s(\mu_i) \neq s(\mu_j)$ for $i\neq j$.  

We say that  {\it no cycle has an exit} meaning that no edge exits a cycle (besides the edges of the cycle itself, of course).

The $C^*$-algebra of two disconnected graphs is the direct sum of the $C^*$-algebras of each connected component. It is then easy two bring the result from connected graphs to general graphs. So to ease the notations, we always consider connected graphs in this article.


\begin{definition}
    We define for $v,w \in  G^0$, 
    $$ vG^*w = \{\mu \in G^*| s(\mu) = v, r(\mu) = w\}. $$
    We also define the set of singular vertices which are the sinks and the infinite emitters :
    $$ G^0_{sing} =\{v\in G^0||s^{-1}(v)|\in\{0,\infty\}\} $$ 
    We define an infinite path as an infinite sequence of edges $ e_0...e_ne_{n+1}...$ such that $r(e_{i})=s(e_{i+1}) $ for all $i$. By $G^{\infty}$, we call the set of all infinite paths in $G$. And finally, we define the boundary path space by 
    $$ \partial G = G^{\infty} \cup \{\mu \in G^*| r(\mu)\in G^0_{sing}\}.$$
\end{definition}


\begin{definition}
    For  $\mu\in G^*$ and $\nu\in \partial G$ such that $r(\mu) = s(\nu)$, we define the concatenation of paths $\mu\nu$ as the path of $\partial G$ that starts with $\mu$ and continues with $\nu$.
\end{definition}

We define now the topology on the boundary path space.

\begin{definition}
    For $\mu\in G^*$, the \textit{cylinder set} of $\mu$ is the set of all paths in $\partial G$ that starts with $\mu$ : 
    $$Z(\mu) = \{\mu\tilde{\nu}  \in \partial G| \tilde{\nu}\in r(\mu)\partial G\}$$
    with $r(\mu)\partial G = \{\tilde{\nu}\in \partial G| s(\tilde{\nu})=r(\mu)\}$ .

    For a finite subset $F\subseteq r(\mu)G^1 $, we define
    $$ Z(\mu\backslash F) = Z(\mu)\backslash\cup_{e\in F} Z(\mu e).$$
\end{definition}

\begin{proposition}[\cite{webster2014path}, Theorem 2.1]
\label{topo_G}
    The boundary path space $\partial G$ is a locally compact Hausdorff space with the topology given by the basis
    $$ \{Z(\mu \backslash F)|\mu\in G^*, F\subset r(\mu)G^1 \text{ finite}\}. $$
\end{proposition}

\begin{definition}
    For $n\in \N$, we define 
    $$ \partial G^{\geq n} = \{x\in\partial G| |x| \geq n\}.$$ 
\end{definition}

As a union of open subsets is open, we have the easy following proposition.

\begin{proposition}
     $\partial G^{\geq n} = \cup_{\mu\in G^n}Z(\mu)$ is an open subset of $\partial G$.
\end{proposition}


\begin{definition}
    We define the shift map  $\sigma_G : \partial G^{\geq 1}\to\partial G$ given by  
    $$ \begin{array}{ll}
        \sigma(x_1x_2x_3...) &=  x_2x_3 ... \text{ for } x_1x_2x_3...\in \partial G^{\geq 2} \\
         \sigma(e)  &= r(e) \text{ for } e\in \partial G \cap G^1 .
    \end{array}$$
\end{definition}

\begin{remark}
    For $n\in \N$, the $n^{th}$ composition of $\sigma$ is $\sigma^n : \partial G^{\geq n} \to \partial G$.
    When we write $\sigma^n(x)$, we will assume $x\in \partial G^{\geq n} $.
\end{remark}

The following proposition is from \cite{webster2014path}, we give a proof here for the reader convenience.

\begin{proposition}
    For all $n\in \N$, $\sigma^n : \partial G^{\geq n} \to \partial G$ is a local homeomorphism.
\end{proposition}

\begin{proof}
    Let $x$ be in $\partial G^{\geq n}$. Since $\partial G^{\geq n} = \cup_{\mu\in G^n}Z(\mu)$, there exists $\mu \in  G^n$ such that $x\in Z(\mu)$. 

    We have $Z(\mu) = \{\mu\tilde{\nu}  \in \partial G | \tilde{\nu}\in r(\mu)\partial G\}$ is open and $\sigma^n(Z(\mu))=Z(r(\mu))$ is also open.

    $$ \begin{array}{llll}
        \sigma^n|_{Z(\mu)} : & Z(\mu) &\to &Z(r(\mu))\\
         & \mu \nu &\mapsto &\nu
    \end{array}$$

    is a bijection.

    Let $Z(\nu\backslash F) $ be an open subset of $ Z(r(\mu))$.  Then $$ (\sigma^n|_{Z(\mu)})^{-1}(Z(\nu\backslash F))=Z(\mu\nu\backslash F)$$ is open.

    Let $Z(\nu\backslash F) $ be an open subset of $ Z(\mu)$.  Then $$ \sigma^n|_{Z(\mu)}(Z(\nu\backslash F))=Z(\sigma^n(\nu)\backslash F)$$ is open. (Note that $r(\nu)=r(\sigma^n(\nu))$ so $Z(\sigma^n(\nu)\backslash F)$ is well define).

    So $\sigma^n$ is a local homeomorphism.
\end{proof}

\subsection{Groupoids and Graphs}

\begin{definition}[\cite{brownlowe2017graph}]
    Let $G$ be a grpah. We define

    $$ \G = \{(x,m-n,y)| x,y\in \partial G, m,n \in \N, \sigma^m(x) = \sigma^n(y)\}$$
with product $(x,k,y)(w,l,z) = (x,k+l,z)$ if $y=w$ and undefined otherwise, and inverse given by $(x,k,y)^{-1} = (y,-k,x)$.

With these operations $ \G$ is a groupoid \cite[lemma 2.4]{kumjian1997graphs}.

The unit space $\G^0= \{(x,0,x)| x \in \partial G\}$ is identified with $\partial G$.

The range map and the source map $r,s : \G \to \partial G$ are given by $r(x,k,y) = x$ and $s(x,k,y) = y$.
\end{definition}

\begin{definition}

A topological groupoid $\G$ is \textit{étale} if the range map $r : \G\to \G$ is a local homeomorphism.
    
\end{definition}

\begin{definition}
    \label{topo}
    For $m,n\in \N$ and $U$ an open subset of $\partial G^{\geq m}$ such that the restriction of $\sigma^m$ to $U$ is injective, $V$ an open subset of $\partial G^{\geq n}$ such that the restriction of $\sigma^n$ to $V$ is injective, and that $\sigma^m(U) = \sigma^n(V)$, we define
    $$ Z(U,m,n,V) = \{(x,k,y)\in\G| x\in U, k=m-n,y\in V,\sigma^m(x) = \sigma^n(y)\} .$$
\end{definition}

\begin{proposition}[\cite{kumjian1997graphs}, Proposition 2.6]
    $\G$ is a locally compact, Hausdorff, étale topological groupoid with the topology generated by the basis consisting of the sets $Z(U,m,n,V)$. 
    
\end{proposition}

\begin{definition}
    For $\mu,\nu\in G^*$ with $r(\mu)=r(\nu)$, we define 
    $$ Z(\mu,\nu) = Z(Z(\mu),|\mu|,|\nu|,Z(\nu)). $$
\end{definition}

\begin{proposition}[\cite{brownlowe2017graph}]
    For $\mu,\nu\in G^*$ with $r(\mu)=r(\nu)$, $Z(\mu,\nu)$ is compact and open and the subspace topology on $\partial G$ coming from the topology on $\G$ agrees with the topology on $\partial G$ given in Proposition \ref{topo_G}.
\end{proposition}

\begin{proposition}[\cite{yeend2007groupoid} Proposition 6.2]
    $\G$ is topologically amenable with this topology.
\end{proposition}

So the reduced and universal $C^*$-algebras of $\G$ are equal, and we denote $C^*(\G) $ this $C^*$-algebra.

\begin{proposition}[\cite{brownlowe2017graph} Proposition 2.2]
\label{Iso}
    Let $G$ be a graph. Then there is a unique isomorphism $\pi : C^*(G) \to C^*(\G)$ such that $\pi(p_v) = 1_{Z(v,v)}$ for all $v\in G^0$ and $\pi(s_e)=1_{Z(e,s(e))}$ for all $e\in G^1$.
\end{proposition}

\begin{definition}
    
 Let $G$ be graph. An \textit{infinite backward chain} is  an infinite sequence of distinct edges $ ...e_{-1}e_{0}e_{1}...$   such that $r(e_{i})=s(e_{i+1})$ for all $i$ and each edge has a predecessor (i.e. for all $e_i$ there exits and edge $e_{i-1}$ in the infinite backward chain such that $r(e_{i-1})=s(e_{i})$).

 It can be on the form

 \begin{center}
    
 \begin{tikzpicture}[node distance={15mm}, thick, main/.style = {draw, fill, circle, inner sep=1.5pt},mainLarge/.style = {draw, circle, minimum size=30mm}]

 \def\Unit{1}

\node (P1) at (0*\Unit,0*\Unit){};
\node[main, label=below:$ $,black] (P2) at (1*\Unit,0*\Unit){};
\node[main, label=below:$ $,black] (P3) at (2*\Unit,0*\Unit){};
\node[main, label=below:$ $,black] (P4) at (3*\Unit,0*\Unit){};
\node[main, label=below:$ $,black] (P5) at (4*\Unit,0*\Unit){};
\node (P6) at (5*\Unit,0*\Unit){};

\path[-angle 90,font=\scriptsize]
(P1) edge [dotted,"",black]   (P2)

(P2) edge ["$e_{-1}$",black]   (P3)
(P3) edge ["$e_{0}$",black]   (P4)
(P4) edge ["$e_{1}$",black]   (P5)

(P5) edge [dotted,"",black]   (P6)
;

\end{tikzpicture}
 
\end{center}

or

\begin{center}
    
 \begin{tikzpicture}[node distance={15mm}, thick, main/.style = {draw, fill, circle, inner sep=1.5pt},mainLarge/.style = {draw, circle, minimum size=30mm}]

 \def\Unit{1}

\node (P1) at (0*\Unit,0*\Unit){};
\node[main, label=below:$ $,black] (P2) at (1*\Unit,0*\Unit){};
\node[main, label=below:$ $,black] (P3) at (2*\Unit,0*\Unit){};
\node[main, label=below:$ $,black] (P4) at (3*\Unit,0*\Unit){};

\path[-angle 90,font=\scriptsize]
(P1) edge [dotted,"",black]   (P2)

(P2) edge ["$e_{-1}$",black]   (P3)
(P3) edge ["$e_{0}$",black]   (P4)

;

\end{tikzpicture}
 
\end{center}

 
    
    


    





 

\end{definition}


    

\begin{definition}

 Let $\G$ be a groupoid with locally compact unit space $\G^{0} $. For $x,y \in \G^{0}$ it is $\G_x^y=\{\gamma \in \G| s(\gamma) = x, r(\gamma) = y \}$,  and call $\G_x^x$, the isotropy subgroup of $\G$ at $x$.
\end{definition}

\begin{definition}

We define the following relation on $\G^{0}$: for $x,y \in \G^{0}$ we have $x\sim y$ if $\G_x^y\neq 0$. The relation $\sim$ is an equivalence relation.
The equivalence classes are called orbits of $\G$ (inside $\G^{0})$. The class of the element $x\in \G^{0}$ is denoted $[x]$ and the cardinal of this class is denoted $|[x]|$.
A point $x\in \G^{0}$ is said to be $periodic$ if its orbit is finite.
\end{definition}

\begin{definition}

 A group is called maximally almost periodic (MAP) if it has a separating family of finite-dimensional representations.
\end{definition}

\begin{theorem}[\cite{shulman2023rfd}]
\label{ThmRFD}
    Let $\G$ be an amenable étale groupoid with locally compact unit space $X$. If $C^*(\G)$ is $RFD$ then $X$ admits a dense set of periodic points; and if $X$ admits a dense set of periodic points with $MAP$ isotropy subgroups, then $C^*(\G)$ is $RFD$.
\end{theorem}

\section{Main result}

\begin{proposition}
 \label{map}
    For any graph, all the isotropy groups of the graph groupoid are MAP.
\end{proposition}

\begin{proof}
    Let $G$ be a graph. Let $(x,0,x)$ be in $\G^0$. 
    
    $$\G^x_x = \{\gamma \in \G| s(\gamma) = x, r(\gamma) = x \}.$$ 

    Either $x$ has no period (no shift brings back to x) and then $\G^x_x = \{(x,0,x)\}$ or $x$ has a minimal period (noted $l$) and then $\G^x_x =\{(x,l.k,x)|k\in \Z\}$. In any case

    $$\G^x_x \subseteq\{(x,k,x)|k\in \Z\} \cong \Z.$$ 
    
    $\Z$ is abelian so it implies that the irreducible representations of $\G^x_x$ are of dimension at most one. 
    So the isotropy groups are MAP.
\end{proof}

\begin{remark}

    The topology on the unit space (the boundary path space) has basis 
    $$ \{Z(\mu\backslash F)|\mu \in G^*,F \text{ finite subset of } r(\mu)G^1\}. $$

    According to Theorem \ref{ThmRFD}, when we want to prove that $C^*$-algebra of a graph is not RFD, it is sufficient to find a set of the form $Z(\mu)$ not containing periodic points.
    



\end{remark}

\begin{lemma}
\label{backward}
  Let $G$ be a  graph. If $G$ contains an infinite backward chain, then $C^*(G)$ is not RFD.
\end{lemma}

\begin{proof}
    
    Let $z$ be an infinite backward chain. 
    If $z$ has an end, we write it $z= ...\mu_{-1}\mu_0$ with $\mu_0$ the last edge. 
    
    If $z$ has no end, we choose one edge in $z$ that we call $\mu_0$ and we write $z =...\mu_{-1}\mu_0\mu_{1}... $. 
    
    
    We will show that $Z(\mu_{0})$ does not contain a periodic point. 
    

    
    
     
     
     


     Let $x$ be in $Z(\mu_{0})$. Since $ \{(\mu_{-1} x,1,x)\}\subseteq \G_{x}^{\mu_{-1} x}$, then  $\mu_{-1}x\in \G^0$ is in the orbit of $x\in\G^0$. 
    
    Similarly, we have that all the paths $\mu_{-1}x$, $\mu_{-2}\mu_{-1}x$, $\mu_{-3}\mu_{-2}\mu_{-1}x$,... seen as element of $\G^{0}$ are all in the orbit of $(x,0,x)$. So the orbit is infinite.     
     
     So the orbit is infinite and $(x,0,x)$ is not periodic. So $Z(\mu_0)$ does not contain a periodic point. 
     
     And so $\G^{0}$ does not contain a dense  subset of periodic point. By Theorem \ref{ThmRFD}, $C^*(\G)$ is not RFD. By Proposition \ref{Iso}, $C^*(\G) \cong C^*(G)$. And so $C^*(G)$ is not RFD.
\end{proof}


\begin{lemma}
\label{receiver}
    Let $G$ be a graph. If $G$ contains an infinite receiver, then $C^*(G)$ is not RFD.
\end{lemma}

\begin{proof}
    Let $G$ be a graph with an infinite receiver vertex called $p_0$. We label $\mu_1, \mu_2,\dots$ the edges whose range is $p_0$. 
    For each element $x$ of $Z(p_0)$, we have that the orbit of $x$ contains 

    $$ \{\mu_1 x,\mu_2 x,\dots\} .$$

    So the orbit of $x$ is infinite and the open set $Z(p_0)$ does not contain a periodic element. By Theorem \ref{ThmRFD}, $C^*(\G)$ is not RFD. By Proposition \ref{Iso}, $C^*(\G) \cong C^*(G)$. And so $C^*(G)$ is not RFD.
    
\end{proof}

\begin{lemma}[\cite{bellier2025rfd} Lemma 4.1]
\label{MyNecessaryCondition} Let $G$ be  a graph. If $C^*(G)$ contains a cycle with an exit, then $C^*(G)$ is not RFD.
\end{lemma}

\begin{lemma}
\label{lastCond}
    Let $G$ be a graph such that there exists $v\in G^0$, such that there is no finite path from $v$ to a sink or a cycle or an infinite emitter.
    Then  $C^*(G)$ is not RFD.
\end{lemma}

\begin{proof}
    Suppose that there exists a vertex $v\in G^0$, such that there is no finite path from $v$ to a sink, a cycle nor an infinite emitter. Then for any $y\in Z(v)$, $y$ is an infinite path. Since $v$ does not lead to a cycle, all the edges of $y$ are distinct. Hence the orbit of $y$ contains the infinite set $ \{\sigma^n(y),n\in \N\}$. 
    So the orbit of $y$ is infinite and the open set $Z(v)$ does not contain a periodic element. By Theorem \ref{ThmRFD}, $C^*(\G)$ is not RFD. By Proposition \ref{Iso}, $C^*(\G) \cong C^*(G)$. And so $C^*(G)$ is not RFD.
\end{proof}

\begin{lemma}
\label{konig}
    Let $G$ be a directed graph with no cycle, no infinite receiver and a vertex $w_0$ such that each  vertex of the graph $G$ belong to a path that ends in $w_0$. 
    
    If $G$ contains infinitely many vertices, then $G$ contains an infinite backward chain.
\end{lemma}

\begin{remark}
    This lemma is an adaptation of the König's lemma to some directed graphs.
\end{remark}

\begin{proof}
We will build the infinite backward chain by recursion.

We start with the vertex $w_0$. It can be seen as a  path of length $0$.

By assumption, there are infinitely many vertices in $G$ and each of them is on a path that ends in $w_0$. So there are infinitely many paths that end in $w_0$. But $G$ does not contain infinite receiver so there are only finitely many edges with range $w_0$. 

By the pigeonhole principle, there is at least one edge that starts from a vertex $w_1$ such that there are infinitely many vertices of the graph $G$ that belong to a path that ends in $w_1$. Since $G$ has no cycle, $w_1$ is different from $w_0$.

With the path from $w_1$ to $w_0$, we have a backward path of length $1$.

Now, assume we have build  a backward path of length $i$ for $i\in \N$ that starts from the vertex $w_i$ and such that there are infinitely many vertices of the graph $G$ that belongs to a path that ends in $w_i$. Since $G$ has no cycle, all edges of this path are distinct.

Since there is no infinite receiver, by the pigeonhole principle, there is at least one edge that starts from a vertex $w_{i+1}$ such that there are infinitely many vertices of the graph $G$ that belongs to a path that ends in $w_{i+1}$. Since $G$ has no cycle, this new edge from $w_{i+1}$ to $w_i$ is distinct from all other edges of the current path from $w_{i}$ to $w_0$.

With the path from $w_{i+1}$ to $w_0$ passing by all $w_i$ previously chosen, we have a backward path of length $i+1$ with all its edges distinct.

By recurrence, we build a sequence $(x_i)_i$ of backward paths ending at $w_0$ with distinct edges and length $i$ for the $i^{th}$ term of the sequence.


Since for each $i\in\N$, $x_{i+1}$ is an extension of $x_i$, the limit is a well define infinite backward chain in $G$.
\end{proof}

\begin{lemma}
\label{finitepaths}
    Let $G$ be a directed graph with no cycle, no infinite receiver and a vertex $w_0$.
    If $G$ has finitely many vertices then $G$ has finitely many paths that end in $w_0$.
\end{lemma}

\begin{proof}
     First, we use the fact that the length of a path without cycle is the number of its distinct vertices minus one. Since $G$ has no cycle, each path has no cycle. And since $G$ has finitely many vertices, say $N+1$, the length of the paths is bounded by $ N$.

     Second, $G$ has no infinite receiver and has finitely many vertices. So for each vertex $p$, $\#\{e\in G^1|r(e)=p\}<\infty$. Since there are finitely many vertices, we can take  $M = \max_{p\in G^0}\#\{e\in G^1|r(e)=p\}<\infty$.

     To see how these bounds imply a finite number of paths that end in $w_0$, we index the incoming edges of each vertex. Meaning that for each vertex $v_i$ which is not a source, we have $r^{-1}(v_i) = \{e_1,...,e_{k_i}\}$ with $0<k_i< M$ for all $v_i$.




Then the number of paths in $G$ is bounded by $\sum_{i=0}^N M^i$.

So $G$ has finitely many paths that end in $w_0$.
\end{proof}

\begin{theorem}
\label{general}
    Let $G$ be a graph. Then

    $C^*(G)$ is RFD if and only if $G$ has all of the following properties
    
    \begin{itemize}
        \item[a)] no infinite receiver
        \item[b)] no cycle with an exit
        \item[c)] no infinite backward chain
        \item[d)] for each $v\in G^0$, there is a finite path from $v$ to a sink or a cycle or an infinite emitter.
    \end{itemize}
\end{theorem}

\begin{remark}
    Before the proof, we illustrate that all those conditions are independent of each other. The figure \ref{figA} shows a graph satisfying the conditions b), c) and d) but not a). The figure \ref{figB} shows a graph satisfying the conditions a), c) and d) but not b). The figure \ref{figC} shows a graph satisfying the conditions a), b) and d) but not c). The figure \ref{figD} shows a graph satisfying the conditions a), b) and c) but not d).

\begin{figure}
\centering
\begin{subfigure}{0.4\textwidth}
    \begin{tikzpicture}[node distance={15mm}, thick, main/.style = {draw, fill, circle, inner sep=1.5pt}, second/.style = {draw, circle}] 

\node[main, label=left:] (v) at (0,0){}; 
\node[main, label=left:] (inf1) at (-1,1){};
\node[main, label=left:] (v1) at (-1,0){};
\node[main, label=left:] (v2) at (-1,-1){};
\node[main, label=left:] (v3) at (0,-1){};
\node[main, label=left:] (inf2) at (1,-1){};
\node[main, label=left:] (v4) at (1,1){};
\node[main, label=left:] (v5) at (2,2){};
\node[main, label=left:] (inf3) at (3,3){};
\node[main, label=left:] (v6) at (2,0){};
\node[main, label=left:] (v7) at (3,1){};

\path[-angle 90,font=\scriptsize]

(inf1) edge [dotted,""]   (v)
(v1) edge [""]   (v)
(v2) edge [""]   (v)
(v3) edge [""]   (v)
(inf2) edge [dotted,""]   (v)

(v) edge [""]   (v4)
(v4) edge [""]   (v6)

(v4) edge [""]   (v5)
(v5) edge [""]   (v7)
(v5) edge [dotted,""]   (inf3);

\end{tikzpicture}
    \caption{}
    \label{figA}
\end{subfigure}
\hfill
\begin{subfigure}{0.4\textwidth}
    \begin{tikzpicture}[node distance={15mm}, thick, main/.style = {draw, fill, circle, inner sep=1.5pt}, second/.style = {draw, circle}] 
    
\node[main, label=left:] (v) at (0,0){}; 
\node[main, label=left:] (v1) at (-1,0){};
\node[main, label=left:] (v2) at (-1,-1){};
\node[main, label=left:] (v3) at (0,-1){};
\node[main, label=left:] (v4) at (1,1){};
\node[main, label=left:] (v5) at (2,2){};
\node[main, label=left:] (inf3) at (3,3){};
\node[main, label=left:] (v6) at (2,0){};
\node[main, label=left:] (v7) at (3,1){};

\path[-angle 90,font=\scriptsize]

(v) edge [""]   (v1)
(v1) edge [dotted,""]   (v2)
(v2) edge [""]   (v3)
(v3) edge [""]   (v)

(v) edge [""]   (v4)
(v4) edge [""]   (v6)
(v6) edge [out=0,in=-60,distance=2cm,""]   (v6)

(v4) edge [""]   (v5)
(v5) edge [""]   (v7)
(v5) edge [""]   (inf3);

\end{tikzpicture} 
    \caption{}
    \label{figB}
\end{subfigure}
\hfill
\begin{subfigure}{0.4\textwidth}
    \begin{tikzpicture}[node distance={15mm}, thick, main/.style = {draw, fill, circle, inner sep=1.5pt}, second/.style = {draw, circle}] 
\node[main, label=left:] (v) at (0,0){}; 
\node[main, label=left:] (inf1) at (-1,-1){};
\node[main, label=left:] (v4) at (1,1){};
\node[main, label=left:] (v5) at (2,2){};
\node[main, label=left:] (inf3) at (3,3){};
\node[main, label= left:] (v6) at (2,0){};
\node[main, label=left:] (v7) at (3,1){};
\node[main, label=left:] (vinf1) at (2.5,0.5){};
\node[main, label=left:] (vinf2) at (1.5,-0.5){};

\path[-angle 90,font=\scriptsize]

(v) edge [""]   (v4)
(inf1) edge [dotted,""]   (v)

(v4) edge [""]   (v5)
(v4) edge [dotted,""]   (vinf1)
(v4) edge [""]   (v6)
(v4) edge [dotted,""]   (vinf2)

(v5) edge [""]   (v7)
(v5) edge [dotted,""]   (inf3);

\end{tikzpicture} 
    \caption{}
    \label{figC}
\end{subfigure}
\hfill
 \begin{subfigure}{0.4\textwidth}
    \begin{tikzpicture}[node distance={15mm}, thick, main/.style = {draw, fill, circle, inner sep=1.5pt}, second/.style = {draw, circle}] 
    
\node[main, label=left:] (v) at (0,0){}; 
\node[main, label=left:] (v1) at (1,1){};
\node[main, label=left:] (v2) at (2,2){};
\node[main, label=left:] (inf1) at (3,3){};

\path[-angle 90,font=\scriptsize]

(v) edge [""]   (v1)

(v1) edge [""]   (v2)
(v2) edge [dotted,""]   (inf1);

\end{tikzpicture} 
    \caption{}
    \label{figD}
 \end{subfigure}
        
\caption{Independance of the four conditions of the theorem \ref{general}.}
\label{ContrEx}
\end{figure}
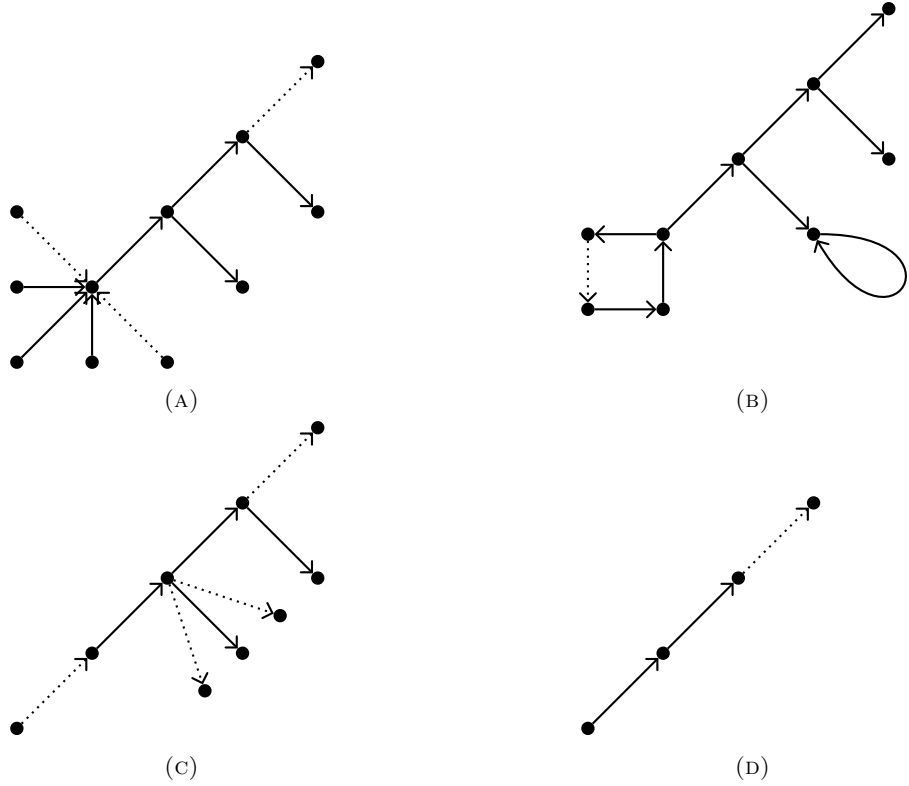

\end{remark}

\begin{proof}
    We prove the \textit{if} part by contraposition.
    
    If we have an infinite receiver, by Lemma \ref{receiver}, $C^*(G)$ is not RFD. 
    
    If we have a cycle with an exit, by Lemma \ref{MyNecessaryCondition}, $C^*(G)$ is not RFD.
    
    If we have an infinite backward chain. By Lemma \ref{backward}, $C^*(G)$ is not RFD.
    
    Finally suppose that there exists a vertex $v\in G^0$, such that there is no finite path from $v$ to a sink, a cycle nor an infinite emitter, by  Lemma \ref{lastCond}, $C^*(G)$ is not RFD.

    Now, for the \textit{only if} part, let $\mu$ be a finite path in $G$ and $F$ a finite subset of $r(\mu)G^1$. We will prove that if $Z(\mu\setminus F)$ is not empty, then it contains at least one element with finite orbit. 

    First, suppose $r(\mu)$ is not a sink. 
    
    If $F=r(\mu)G^1$ then $Z(\mu\setminus F)=\emptyset$. If $F \neq r(\mu)G^1$, we take one $\nu\in  r(\mu)G^1 \setminus F$. 

    From the hypothesis, there is a finite path from $r(\nu)$ to a sink, a cycle or an infinite emitter. We call this path $\tilde{y}$. 

    Suppose $r(\tilde{y})$ belongs to a cycle $a=(e_1,\dots,e_n)$ with $s(e_1)=r(\tilde{y})$ and $r(e_n)=s(e_1)$. Let $y = \mu\nu\tilde{y}a^{\infty}$ the path that goes from $\mu$ to $r(\tilde{y})$ and then turn infinitely around the cycle. We have that $y$ is in $\partial G$ and in $Z(\mu\backslash F)$. And $y$ contains finitely many vertices.

    If $r(\tilde{y})$ is a sink or and infinite emitter, then we set $y=\mu\nu\tilde{y}$. And again, $y$ is in $\partial G$, in $Z(\mu\backslash F)$ and contains finitely many vertices.

    Now suppose $r(\mu)$ is a sink. Then we define $y= \mu$. Then $y$ is in $\partial G$, in $Z(\mu\backslash F)$ and contains finitely many vertices.




    Then in all cases, $y\in \partial G$ and $y\in Z(\mu\backslash F)$ and $y$ contains finitely many vertice.
    We will show that the orbit of $y$ is finite. For this, it is sufficient to show that only finitely many paths reach $y$. Since $y$ contains finitely many vertices, we show that each of them receives finitely many paths. 

    Let $w_0$ be one of the finitely many vertices of $y$.

    We consider the subgraph $H$ of $G$ that contains the edges and the vertices from the paths that ends in $w_0$. 

    $H$ is a directed graph with no infinite receiver since there is none in $G$. By construction, $w_0$ is such that all vertices and edges of the graph $H$ belongs to a path that ends in $w_0$. 

We distinguish three cases.

Case 1 : $w_0$ is a source. Then there are no path  that reach $w_0$.

Case 2 : $w_0$ is not a source and $w_0$ is not on a cycle. Then there is no cycle in $H$ since it would mean there is a cycle with an exit and $G$ has none of them.
    
If $H$ contains infinitely many vertices, then by Lemma \ref{konig}  $H$ contains an infinite backward chain. 

But $H$ does not contain an infinite backward chain because $G$ does not. 

    So $H$ does not contain infinitely many vertices. So by Lemma \ref{finitepaths} $H$ contains finitely many paths that end in $w_0$.

    Then there are finitely many paths in $G$ that reach $w_0$.

Case 3 : $w_0$ is not a source and $w_0$ is on a cycle. Then we can assume this cycle has $n$ vertices, $x_1=w_0,x_2,...,x_n$. For the vertex $x_i$, we consider the subgraph $H_i$ of $G$ that contains the edges and the vertices from the paths that ends in $x_i$ except for the ones coming from the cycle containing $w_0$. (two different cycles can not contain $w_0$ since it would mean that there is a cycle with an exit). Then $H_i$ contains no cycle, again because it would make a cycle with an exit.

Same as for $H$, $H_i$ is a directed graph with no infinite receiver since there is none in $G$. By construction, $x_i$ is such that all vertices and edges of the graph $H_i$ belongs to paths that end in $x_i$. 

So as before, by Lemma \ref{konig}, if $H_i$ contains infinitely many vertices, then   $H_i$ contains an infinite backward chain. 

But $H_i$ is a subgraph of $G$ that does not contain an infinite backward chain. 

    So $H_i$ does not contains infinitely many vertices. So by Lemma \ref{finitepaths} $H_i$ contains finitely many paths that end in $x_i$.

Since the cycle has no exit, the number of paths that reach $w_0$ is the sum of the length of the cycle plus the sum of  the number of paths that reach $x_i$ for $i\in \{1,...,n\}$. 

Then there are finitely many paths that reach $w_0$. So in all three cases, finitely many paths reach $w_0$. So finitely many paths reach $y$. So the orbit $[y]$ is finite. So the periodic elements are dense in $\G^{(0)}$.  

By Theorem \ref{ThmRFD}, $C^*(\G)$ is  RFD. By Proposition \ref{Iso}, $C^*(\G) \cong C^*(G)$. So $C^*(G)$ is RFD.    
    
\end{proof}
 
\bibliographystyle{plain}
\bibliography{localbibliography}  











\end{document}